\documentclass[10pt,a4paper]{article}

\usepackage[utf8]{inputenc}
\usepackage[cyr]{aeguill}
\usepackage[english,francais]{babel}
\usepackage{amsmath}
\usepackage{amsthm}
\usepackage{amsfonts}
\usepackage{amssymb}
\usepackage{amsbsy}
\usepackage{bbm}
\usepackage{mathrsfs}
\usepackage{enumerate}
\usepackage{setspace}
\usepackage{multicol}

\usepackage{graphicx, psfrag, pstricks}
\usepackage{color}
\usepackage{url}
\usepackage{hyperref}

\def \eps{\varepsilon}
\def \a{\alpha}

\def \s{\sigma}

\def \k{\xi}

\def \G{\Gamma}

\def \UU{\mathcal U}
\def \R{\mathbb R}

\def \N{\mathbb N}

\def \B{\mathcal B}
\def \P{\mathcal P}

\def \st{\; | \;}

\newcommand*{\LargerCdot}{\raisebox{-0.2ex}{\scalebox{1.4}{$\cdot$}}}

\newcommand{\closeunder}[2]{\underset{\raise\hspace{-0.8mm}box{1mm}{\ensuremath{#2}}}{#1}}
\newcommand*{\quot}[2]%
{\ensuremath{%
    #1/\!\raisebox{-.65ex}{\ensuremath{#2}}}}

\newtheorem{theo}{Theorem}[section]
\newtheorem{prop}[theo]{Proposition}

\newtheorem{lem}[theo]{Lemma}

\newtheorem{cor}[theo]{Corollary}
\newtheorem{df}[theo]{Definition}
\newtheorem{rmq}[theo]{Remark}

	{\par\parindent0pt\medskip\textbf{Theorem\ifx&#1&\else\space#1\fi.}%
	  \itshape\par}{\par\medskip}

	{\par\parindent0pt\medskip\textbf{Corollary\ifx&#1&\else\space#1\fi.}%
	  \itshape\par}{\par\medskip}

\title{ Fibred coarse embeddability of box spaces and proper isometric affine actions on $L^p$ spaces}
\author{S. Arnt}
\date{}

\begin{document}

\renewcommand{\proofname}{Proof}
\renewcommand\refname{References}
\renewcommand\contentsname{Table of contents}
\renewcommand\chaptername{Chapter}
\renewcommand{\abstractname}{Abstract}
\renewcommand{\thefootnote}{\*}

\maketitle

\begin{abstract}
We show the necessary part of the following theorem : a finitely generated, residually finite group has property $PL^p$ (i.e. it admits a proper isometric affine action on some $L^p$ space) if, and only if, one (or equivalently, all) of its box spaces admits a fibred coarse embedding into some $L^p$ space. We also prove that coarse embeddability of a box space of a group into a $L^p$ space implies property $PL^p$ for this group.
\end{abstract}

\footnotetext{The research of the author was supported by grant 20-149261 of Swiss SNF and the ANR project GAMME.}

\section{Introduction}
%

The notion of fibred coarse embeddings into Hilbert space, which generalizes the notion of coarse embeddings, has been introduced by Chen, Wang and Yu in \cite{cwy} to provide a tool for the study of the maximal Baum-Connes conjecture. They proved in this paper that any metric space with bounded geometry admitting a fibred coarse embedding into a Hilbert space satisfies the maximal coarse Baum-Connes conjecture. In \cite{cww}, Chen, Wang and Wang characterized the Haagerup property in terms of fibred coarse embedding into Hilbert space: in fact, they showed that a finitely generated, residually finite group has the Haagerup property if, and only if, one of its box space admits a fibred coarse embedding into a Hilbert space. The goal of this note is to extend this result to the class of $L^p$ spaces (for a fixed $p \geq 1$).

\begin{theo}\label{maintheo}
 Let $\G$ be a finitely generated, residually finite group, $(\G_i)_{i\in N^*}$ be a nested sequence of finite index normal subgroups of $\G$ with trivial intersection and $1 \leq p \leq \infty$. Then $\G$ has property $PL^p$ if, and only if, the box space $\Box_{\{\G_i\}}\G$ admits a fibred coarse embedding into a $L^p$ space.
\end{theo}

We also prove the following proposition which extends to $L^p$ spaces a result of Roe in the setting of Hilbert spaces (see \cite{roe}).   

\begin{prop}\label{coarseemb}
 Let $\G$ be a finitely generated, residually finite group, $(\G_i)_{i\in N^*}$ be a nested sequence of finite index normal subgroups of $\G$ with trivial intersection and $1 \leq p \leq \infty$. If the box space $\Box_{\{\G_i\}}\G$ admits a coarse embedding into a $L^p$ space, then $\G$ has property $PL^p$.
\end{prop}

 Theorem \ref{maintheo} and Proposition \ref{coarseemb} can be stated for other classes of Banach spaces instead of $L^p$ spaces. In fact, the proof of the necessary condition (see Proposition \ref{fceplp}) and the proof of Proposition \ref{coarseemb} only uses the fact that the class of $L^p$ spaces (for a fixed $1\leq p < \infty$) is a class $\B$ of Banach spaces satisfying the following properties: \smallskip
 \begin{enumerate}
  \item $\B$ is closed under taking some particular normed finite powers i.e.: \smallskip \\
  for every $n \in \N^*$ and every $B \in \B$, there exists a norm $N$ on $\R^n$ such that:
  \begin{itemize}
   \item there exists $c \geq 0$ such that, for all $K,K' \geq 0$ the $n$-cube $\{x \in \R^n \st K\leq x_i \leq K'\}$ is contained in the annulus $\{x\in \R^n \st cK \leq N(x) \leq cK'\}$ - or, in other words, for all $x \in \R^n$, if the components of $x$ are controlled below by $K$ and above by $K'$ then so does $\frac{1}{c}N(x)$; \smallskip
   \item the Banach space $B^n$ endowed with the norm $\|\LargerCdot\|=N(\|\pi_{1}(\LargerCdot)\|_{B},...,\|\pi_{n}(\LargerCdot)\|_{B})$ belongs to $\B$ (where $\pi_i$ is the canonical projection of $B^n$ on its $i$-th factor).
  \end{itemize}
In the $L^p$ case, for $n \in \N^*$, the norm of $\ell_p^n=\ell^p(\{1,...,n\})$ fits, and $c=n^{\frac{1}{p}}$.  \smallskip
  \item $\B$ is closed under ultraproducts (see Definition \ref{ultraproduct}). \\
  In the $L^p$ case, the stability by ultraproduct is a result due to Krivine (see \cite{kri} Theorem 1 and its application p.17).
 \end{enumerate}

For a class of Banach spaces $\B$, property $P\B$ is an analog of the Haagerup property viewed with the Gromov's definition of a-T-menability (definition in terms of isometric affine actions, see \cite{gro93} or \cite{checowjoljulval}) where the class of Hilbert spaces is replaced by the class $\B$. One of the motivation in the study of this property is given by a result of Kasparov and Yu in \cite{kasyu} which asserts that groups admitting coarse embeddings into uniformly convex Banach spaces satisfy the Novikov conjecture (in particular, groups having property $P\B$ where $\B$ is a subclass of uniformly convex Banach spaces admit such embeddings). \\


An \emph{isometric affine action} of a group $\G$ on a Banach space $B$ is a morphism $\a$ of $\G$ into the group $\text{Aff}(B)\cap \text{Isom}(B)$ of affine isometric transformations of $B$; such an action can be characterized by the following decomposition:
$$ \a(g)v = \pi(g)v+b(g), \text{ for all }g\in \G,\; v \in B, $$
where $\pi$ is an isometric representation of $\G$ on $B$ and $b$ is a 1-cocycle with respect to $\pi$ i.e., for all $g,h \in \G$, $b(gh)=\pi(g)b(h)+b(g)$. \\
The action $\a$ is said to be \emph{proper} if $\|b(g)\|_B \underset{g \rightarrow \infty}{\longrightarrow}+\infty$. 

\begin{df}\label{pb}
 Let $\B$ be a class of Banach spaces. A (discrete) group $\G$ is said to have \emph{property $P\B$} if there exists a proper isometric affine action of $\G$ on some Banach space $B \in \B$.
\end{df}

Many recent progress has been made in the study of isometric affine actions on Banach spaces, and more particularly in the case of $L^p$ spaces for a fixed $1\leq p \leq \infty$. Bader, Furman, Gelander, Monod studied the relationships between two different generalizations of Kazhdan's property $(T)$, namely property $FL^p$ and property $(T_{L^p})$ in \cite{badfurgelmon}. On the other hand, property $PL^p$, also referred as a-$FL^p$-menability by some authors, is a strong negation of property $FL^p$. Examples of $PL^p$ groups are given by \cite{yuhyp}, where Yu proved that, for a discrete hyperbolic group $\G$, there exists $2 \leq p_0 < \infty$ such that $\G$ has property $PL^p$ for all $p \geq p_0$; or by \cite{cortesval}, where Cornulier, Tessera, Valette showed that the hyperbolic simple Lie group Sp$(n,1)$ has property $PL^p$ for all $p > 4n+2$. We give here an overview of what is known about the links between property $PL^p$ and $PL^q$ for various values of $p$ and $q$:
\begin{center}
\begin{tabular}{clcl}
 (1)& Haagerup (=$PL^2$) &$\Rightarrow$& $PL^p$ for all $0<p\leq \infty$ \\
 (2)& $PL^p$ for some $0<p\leq 2$ &$\Leftrightarrow$& Haagerup \\
 (3)&$PL^p$ for some $p>2$ & $\nRightarrow$& Haagerup \smallskip \\
 (4)&$PL^p$ for some $p>2$ &$\underset{??}{\Rightarrow}$& $PL^q$ for all $q>p$ \\
\end{tabular}
\end{center}
Implication (1) was proved in \cite{chemarval} by Cherix, Martin and Valette for countable discrete groups, using the notion of spaces with measured walls. Equivalence (2) follows from results of Delorme-Guichardet (\cite{gui}, \cite{del}) and Akemann-Walter (\cite{akewal}). See \cite{chadruhag} Corollary 1.5 and Remark 1.6 for proofs and discussions about (1) and (2) in the setting of second countable, locally compact groups. \\
Assertion (3) follows from the fact that a discrete hyperbolic group with property $(T)$ fails the Haagerup property but has $PL^p$ for some $p >2$ by the result of Yu quoted before. We mention that assertion (4) is still an open question which appears in \cite{chadruhag}, Question 1.8. \smallskip \\
Concerning stability, property $PL^p$ (for a fixed $p >2$) is closed under taking closed subgroups, direct sums, amalgamated free products over finite subgroups (see \cite{pil} and \cite{arnt1} for proofs of this result with different approachs) but it is not stable by extension in general. However, using a construction of Cornulier, Stalder and Valette in \cite{wreath}, the author showed in \cite{arn} that property $PL^p$ is closed under wreath product by Haagerup groups. We would like to mention that Haagerup property is stable by amenable extensions, but for property $PL^p$ with $p>2$, it remains an open problem. 

\begin{rmq}
 Notice that unlike in the Hilbert spaces case, property $PL^p$ is no longer equivalent to property $HL^p$ i.e. the existence of a $C_0$ representation on some $L^p$ space which almost has invariant vectors. For instance, a discrete hyperbolic group with property (T) has property $PL^p$ for some $p>2$, but it also has property $(T_{L^p})$ for all $p\geq 1$ (see \cite{badfurgelmon}) which is a strong negation of property $HL^p$.
\end{rmq}

\begin{df}\label{boxspace}
 Let $\G$ be a finitely generated, residually finite group and let $\G_1\unrhd ... \unrhd \G_i \unrhd ...$ be a nested sequence of finite index normal subgroups of $\G$ such that $\bigcap_{i=1}^{\infty} \G_i =\{e\}$. The box space associated with the sequence $\{\G_i\}_{i \in \N^*}$, denoted by $\Box_{\{\G_i\}}\G$ or simply $\Box \G$, is the coarse disjoint union $\bigsqcup_{i=1}^{\infty}\G/\G_i$ of the finite quotient groups, i.e., the disjoint union where each quotient is endowed with the metric induced by the image of the generating set of $\G$, and the distances between the identity elements of two successive quotients are chosen to be greater than the maximum of their diameters.
\end{df}

There is a large spectrum of analytic properties of a group $\G$ which link to geometric properties of its box space $\Box \G$. As in \cite{cww}, we summarize here different correspondances:
\begin{center}
\begin{tabular}{lcl}
 $\G$ amenable &$\Leftrightarrow$& $\Box\G$ Property A \\
 $\G$ Property $(T)$ &$\Leftrightarrow$& $\Box\G$ geometric Property $(T)$ \\
 $\G$ Haagerup & $\Leftrightarrow$& $\Box\G$ fibred coarsely embeddable into Hilbert space \medskip \\
 $\G$ Property $PL^p$ &$\Leftrightarrow$& $\Box\G$ fibred coarsely embeddable into some $L^p $ \\
 $\G$ Property $PL^p$ &$\Leftarrow$& $\Box\G$ coarsely embeddable into some $L^p $ \\
\end{tabular}
\end{center}

The first equivalence was established by Roe in \cite{roe} where Property $A$ is a non-equivariant version of amenability defined by Yu (\cite{yua}) which guarantees coarse embeddability into Hilbert spaces. The second one is due to Willett and Yu in \cite{wilyu} where they introduced the notion of geometric property (T). For a coarse disjoint union of finite graphs, geometric property (T) implies the property of being
an expander. The third equivalence is the result of Chen, Wang and Wang (\cite{cww}) mentioned in the introduction. \\
The last two assertions are proved in the present note. In \cite{roe}, Roe established the last implication in the Hilbert case ($p=2$ case); and notice that the converse implication fails. In fact, on one hand, the free group on two generators has the Haagerup property, and on the other hand, it has property ($\tau$) with respect to some sequences of finite index normal subgroups (see \cite{lub}): hence, the associated box spaces are expanders, which implies that they are not coarsely embaddable into Hilbert space. \\

\emph{Acknowledgement.} The author wish to thank Ana Khukhro, Thibault Pillon and Alain Valette for their comments and helpful discussions during the
elaboration of this note. We also thank Gilles Lancien for raising a question that led to this paper.

\section{Fibred Coarse embeddings into Banach spaces}

We recall here the notion of coarse embedding and the notion of fibred coarse embedding introduced in \cite{cwy} where the notion of Banach spaces replaces the original Hilbert spaces model. \bigskip

\begin{df}\label{dfce}
 Let $(X,d)$ be a metric space and $B$ be a Banach space. A map $f:X\rightarrow B$ is said to be a \emph{coarse embedding} of $X$ into $B$ if there exist two non-decreasing functions $\rho_1$ and $\rho_2$ from $[0, +\infty)$ to $(-\infty,+\infty)$ with $\lim_{r \rightarrow +\infty}\rho_i(r)=+\infty$ for $i=1,2$, such that, for all $x,y \in X$:
 $$\rho_1(d(x,y))\leq \|f(x)-f(y)\|\leq \rho_2(d(x,y)). $$
\end{df}

\begin{rmq}\label{pinfty}
 Every metric space $(X,d)$ admits a coarse embedding into $\ell^{\infty}(X)$ via, for a fixed $x_0 \in X$, the map 
 $$ f:x \rightarrow \{y\mapsto d(x,y)-d(x_0,y)\}.$$
 In fact, $f$ is an isometric embedding. \\
 Moreover, for a finitely generated group $\G$ endowed with the word metric $d$ induced by a finite generating set, the same map $f: g \mapsto \{h\mapsto d(g,h)-d(e_{\G},h)\}$ is a proper cocycle with respect to the left regular representation on $\ell^{\infty}(\G)$. Hence, every finitely generated group has property $PL^{\infty}$.
\end{rmq}

\begin{df}\label{dffce}
 A metric space $(X,d)$ is said to admit a \emph{fibred coarse embedding into} a Banach space $B$, if there exist:
 \begin{enumerate}
  \item a field of Banach spaces $(B_x)_{x \in X}$ over $X$ such that each $B_x$ is affinely isometric to $B$; \\
  \item a section $s:X \rightarrow \bigsqcup_{x\in X} B_x$ (i.e. $s(x) \in B_x$); \\
  \item two non-decreasing functions $\rho_1$ and $\rho_2$ from $[0, +\infty)$ to $(-\infty,+\infty)$ with $\lim_{r \rightarrow +\infty}\rho_i(r)=+\infty$ for $i=1,2$ such that: \\
  for any $r >0$, there exists a bounded subset $K_r \subset X$ for which there exists a ``trivialization''  
  $$ t_C : (B_x)_{x \in C} \rightarrow C\times B $$
  for each subset $C \subset X\smallsetminus K_r$ of diameter less than $r$; that is, a map from $(B_x)_{x \in C}$ to the constant field $C \times B$ over $B$ such that the restriction to the fibre $B_x$ for $x \in C$ is an affine isometry $t_C:B_x \rightarrow B$, satisfying the following conditions:
  \begin{itemize}
   \item[i)] for any $x,y \in C$, $\rho_1(d(x,y))\leq \|t_C(x)(s(x))-t_C(y)(s(y))\|_B \leq \rho_2(d(x,y))$; \\
   \item[ii)] for any two subsets $C_1,C_2 \subset X\smallsetminus K_r$ of diameter less than $r$ with $C_1 \cap C_2 \neq \emptyset$, there exists an affine isometry $t_{C_1C_2}: B \rightarrow B$ such that $t_{C_1}(x)\circ t_{C_2}(x)^{-1}=t_{C_1C_2}$, for all $x \in C_1\cap C_2$.
  \end{itemize}

 \end{enumerate}

\end{df}

\begin{rmq}\label{ceimpfce}
 Let $(X,d)$ be a metric space and $B$ be a Banach space. If $X$ coarsely embeds into $B$ then $X$ fibred coarsely embeds into $B$. In fact, if $f:X \rightarrow B$ is a coarse embedding with control functions $\rho_1,\rho_2$ then a fibred coarse embedding of $X$ into $B$ is given by :
 \begin{enumerate}
  \item the field of Banach spaces $(B_x)_{x \in X}$ where $B_x:=B$ for all $x \in X$; \smallskip
  \item the section $s:x \mapsto f(x) \in B=B_x$; \smallskip
  \item the two control functions $\rho_1$ and $\rho_2$ and for each $r >0$, considering $K_r=\emptyset$, for all $C$ of diameter less than $r$, the ``trivial`` trivialisation given by, for $x \in X$, $t_C(x)=Id_B$ (which satisfies condition $i)$ and $ii)$ since $f$ is a coarse embedding).
  \end{enumerate}
\end{rmq}

The following proposition is proved by Chen, Wang and Wang in \cite{cww} (see Proposition 1.4) in the general setting of fibred coarse embeddings into metric spaces.

\begin{prop}\label{cwwprop}
 Let $\G$ be a finitely generated, residually finite group. If $\G$ acts properly isometrically on a metric space $Y$, then any box space $\Box\G$ admits a fibred coarse embedding into $Y$.
\end{prop}

We can then the reformulate this statement in the context of property $P\B$:

\begin{cor}\label{cwwcor}
 Let $\G$ be a finitely generated, residually finite group and $\B$ a class of Banach spaces. If $\G$ has property $P\B$, then any box space $\Box\G$ admits a fibred coarse embedding into some Banach space $B \in \B$.
\end{cor}

\section{Proof of the main results}

\begin{df}
 Let $\G$ be a finitely generated group and $r$ be a non-negative real. \\
 
 i) Let $X$ be a set. A map $\a: \G\times X \rightarrow X$ is said to be a \emph{$r$-locally action of $G$ on $X$} if:
 \begin{itemize}
  \item for all $g \in \G$ such that $d(e,g)<r$, $\a(g): X \rightarrow X$ is a bijection;
  \item for all $g,h\in \G$ such that $d(e,g)$, $d(e,h)$, $d(e,gh)$ are less than $r$, $$ \a(gh)=\a(g)\a(h).$$
 \end{itemize} \smallskip 

 ii) Let $B$ be a Banach space. A map $\pi: \G\times B \rightarrow B$ is said to be a \emph{$r$-locally isometric representation of $G$ on $B$} if $\pi$ is a $r$-locally action of $\G$ on $B$ and for all $g \in \G$ such that $d(e,g)<r$, $\pi(g): B\rightarrow B$ is a linear isometry.  \smallskip \\
 In this case, a map $b : \G \rightarrow B$ such that, for all $g,h \in \G$ such that $d(e,g)$, $d(e,h)$, $d(e,gh)$ are less than $r$, $\pi(g)b(h)+b(g)=b(gh)$, is called a \emph{r-locally cocycle with respect to $\pi$}. \medskip

 iii) Let $B$ be a Banach space. A map $\a : G\times B \rightarrow B$ is called a \emph{$r$-locally isometric affine action of $\G$ on $B$} if it can be written as $\a(g)\LargerCdot = \pi(g)\LargerCdot +b(g)$ where $\pi$ is a $r$-locally isometric representation and $b$ is a $r$-locally cocycle with respect to $\pi$.
\end{df} \smallskip 

Using the notion of ultrafilters and ultraproducts, one can build a global isometric affine action from a family of $r$-locally isometric affine actions with $r\rightarrow +\infty$. \smallskip \\
Let $\UU$ be a non-principal ultrafilter on $\N^*$ i.e. $\UU$ is a subset of $\P(\N^*)$ stable by intersection such that:
 \begin{itemize}
  \item the empty set $\emptyset$ does not belong to $\UU$,
  \item for all $A,B \in \P(X)$ such that $A \subset B$, $A \in \UU$ implies $B \in \UU$,
  \item for all $A \in \P(X)$, $A\in \UU$ or $X\setminus A \in \UU$.
  \item finite subsets of $\N^*$ do not belong to $\UU$.
 \end{itemize} \smallskip 
 The $\UU$-limit of bounded real valued sequence $(x_r)_{r\in \N^*}$ is the unique $x \in \R$ denoted by $\lim_{\UU}x_r$ such that for all $\eps >0$, the set $\{r \in \N^* \st |x_r -x| \leq \eps \}$ belongs to $\UU$.
 
 \begin{df}\label{ultraproduct}
  Let $(B_r)_{r\in \N^*}$ be a family of Banach spaces and consider the space $\ell^{\infty}(\N^*,(B_r)_{r\in \N^*})$ of sequences $(a_r)_{r \in \N^*}$ satisfying that there exists $K \geq 0$ such that for all $r \in \N^*$, $a_r \in B_r$ with $\|a_r\|_{B_i} \leq K$. \smallskip \\
The \emph{ultraproduct $B_{\UU}$ of the family $(B_r)_{r\in \N^*}$ with respect to a non-principal ultrafilter $\UU$} is the closure of the space $\ell^{\infty}(\N^*,(B_r))/\sim_{\UU}$ endowed with the norm $\|(a_r)\|_{B_{\UU}}:= \lim_{\UU}\|a_r\|_{B_r}$ where, for $(a_r), (b_r) \in \ell^{\infty}(\N^*,(B_r))$, \smallskip 
\begin{center}
 $(a_r) \sim_{\UU} (b_r)$ if, and only if, $\|(a_r)-(b_r)\|_{B_{\UU}}=0$.
\end{center}

 \end{df}

\begin{lem}\label{rlocultraprod}
 Let $\G$ be a finitely generated group, $(B_r)_{r\in \N^*}$ be a family of Banach spaces and $B_{\UU}$ be the ultraproduct of the family $(B_r)$ with respect to a non-principal ultrafilter $\UU$ on $\N^*$. For each $r \in \N^*$, assume that $\G$ admits a $r$-locally isometric affine action $\a_r$ on $B_r$ with $\a_r(g)\LargerCdot = \pi_r(g)\LargerCdot +b_r(g)$.\\
 If, for all $g \in \G$, $(b_r(g))_{r \in \N^*}$ belongs to $B_{\UU}$, then there exists an isometric affine action $\a$ of $G$ on $B_{\UU}$ of the family $(B_r)$ such that $\a(g)\LargerCdot =\pi(g)\LargerCdot +b(g)$ where $\pi$ is an isometric representation of $\G$ on $B_{\UU}$ and $b:G \rightarrow B_{\UU}$ is a cocycle with respect to $\pi$ satisfying, for $g \in \G$:
 $$ b(g)=(b_r(g))_{r \in \N^*}. $$
\end{lem}

\begin{proof}
For $g \in \G$, we define $\pi(g): B_{\UU} \rightarrow B_{\UU}$ by, for $a=(a_r)_{r \in \N^*} \in B_{\UU}$, 
$$ \pi(g)a=(\pi_r(g)a_r)_{r \in \N^*}; $$
and we set $b(g)=(b_r(g))_{r \in \N^*} \in B_{\UU}.$ \\
Let $g,h \in \G$. For all $r \in \N^*$ such that $r > \max(d(e,g), d(e,h), d(e,gh))$, we have, for all $(a_r) \in B_{\UU}$, $\pi_r(g)\pi_r(h)a_r=\pi_r(gh)a_r$ and then the set $\{ r \in \N^* \st \pi_r(g)\pi_r(h)a_r=\pi_r(gh)a_r \}$ belongs to $\UU$. Hence, for all $g,h \in \G$, $\pi(g)\pi(h)=\pi(gh)$.
Now, for $g \in \G$, since for all $r$ large enough, $\pi_r(g)$ is an isometric isomorphism of $B_r$, it follows, by a similar argument, that $\pi(g)$ is an isometric isomorphism of $\B_{\UU}$. \\
Thus, $\pi$ is an isometric representation of $\G$ on $B_{\UU}$. \\
Let $g,h \in \G$. For all $r \in \N^*$ such that $r > \max(d(e,g), d(e,h), d(e,gh))$, we have $b_r(gh)=\pi_r(g)b_r(h)+b_r(g)$. Hence, for all $g,h \in \G$, $b(gh)=\pi(g)b(h)+b(g)$ and then, $b$ is a cocycle with respect to $\pi$.
It follows that the map $\a$ such that $\a(g)\LargerCdot = \pi(g)\LargerCdot +b(g)$ is an isometric affine action of $\G$ on $B_{\UU}$.
\end{proof}

\begin{proof}[Proof of Proposition \ref{coarseemb}]
The case $p=\infty$ is trivial (see Remark \ref{pinfty}). \\
 Let $1\leq p < \infty$. Let $\{\G_n\}_{n \in \N^*}$ be a nested sequence of finite index normal subgroups of $\G$ with trivial intersection such that the associated box space $\Box \G$ admits a coarse embedding $f$ into a $L^p$ space denoted by $B$ with control functions $\rho_1,\rho_2$. \\
 Let $n \in \N^*$ and denote $X_n:=\G/\G_n$. Let us consider the Banach space $\bigoplus_{z\in X_{n}}B$ endowed with the following norm : for a vector $\k=\bigoplus_{z \in X_{n}}\k_z$,
 $$ \|\k\|_p=\left(\sum_{z \in X_{n}}\|\k_z\|_B^p\right)^{\frac{1}{p}}. $$
 For $x\in X_{n}$, we define the following vector of $\bigoplus_{z\in X_{n}}B$:
 $$ \tilde{b}_n(x):= \frac{1}{(\#X_{n})^{\frac{1}{p}}}\bigoplus_{z \in X_{n}}\left(f(zx)-f(z)\right); $$
 and let $\tilde{\s}_n$ be the isometric representation of $X_{n}$ on $\bigoplus_{X_{n}}B$ such that for $\k=\bigoplus_{z \in X_{n}}\k_z$,
 $$ \tilde{\s}_n(x)\k=\bigoplus_{z \in X_{n}}\k_{zx}. $$
 Then $\tilde{b}_n: X_{n} \rightarrow \bigoplus_{X_{n}}B$ is a cocycle with respect to $\tilde{\s}_n$. In fact, since for $x,y,z \in X_{n}$, we have $f(zxy)-f(z)=(f(zxy)-f(zx))+(f(zx)-f(z))$, it follows that: \smallskip
 \begin{center}
  \begin{tabular}{rl}
  $\tilde{b}_n(xy)$&$=\frac{1}{(\#X_{n})^{\frac{1}{p}}}\bigoplus_{z \in X_{n}}\left(f(zxy)-f(zx)\right)+\frac{1}{(\#X_{n})^{\frac{1}{p}}}\bigoplus_{z \in X_{n}}\left(f(zx)-f(z)\right),$ \medskip \\
  $\tilde{b}_n(xy)$&$=\tilde{\s}_n(x)\tilde{b}_n(y)+\tilde{b}_n(x)$.\\
 \end{tabular}
 \end{center}
 Moreover, since $f$ is a coarse embedding, we have, for all $x \in X_{n}$:
 $$ \rho_1(d_{X_n}(x,e))\leq \|\tilde{b}_n(x)\|_p\leq \rho_2(d_{X_n}(x,e)), $$
 where $e$ is the identity element of $X_n$. \medskip\\
 
 Now, for each $r \in \N^*$, choose $n_r$ such that the canonical quotient map $\pi_{n_r}:\G\twoheadrightarrow X_{n_r}$ is $r$-isometric and define $\s_r:=\tilde{\s}_{n_r} \circ \pi_{n_r}$ and $b_r:=\tilde{b}_{n_r} \circ \pi_{n_r}$. Thus, for every $r$, $b_{r}$ is a cocycle with respect to the isometric representation $\s_r$ of $\G$ on $\bigoplus_{X_{n_r}}B$ and we have, for $g \in \G$ such that $d_{\G}(g,e_{\G})<r$:
 $$ \rho_1(d_{\G}(g,e))\leq \|b_r(g)\|_p\leq \rho_2(d_{\G}(g,e_{\G})). \;\; (*)$$
 Let $\UU$ be a non-principal ultrafilter on $\N^*$ and $B_{\UU}$ be the ultraproduct of $\left(\bigoplus_{X_{n_r}}B\right)_{r \in \N^*}$. For each $r$, the map $\a_r$ defined by $\a_r(g)\LargerCdot:=\pi_r(g)\LargerCdot+b_r(g)$ for $g \in \G$, is an isometric affine action. By $(*)$, for all $g \in \G$, $(b_r(g))_{r \in \N^*}$ belongs to $B_{\UU}$. \\
 Hence, by Lemma \ref{rlocultraprod}, there exists an isometric affine action $\a$ of $\G$ on $B_{\UU}$ such that $b: g \mapsto (b_r(g))$ is a cocycle for this action. 
 Moreover, for $g \in \G$, since for all $r$ large enough, $ \rho_1(d_{\G}(g,e))\leq \|b_r(g)\|_p$, we have:
 $$ \rho_1(d_{\G}(g,e))\leq \|b(g)\|_{B_{\UU}}; $$
 hence $\a$ is proper. As the class of $L^p$ spaces is closed under $p$-normed powers and ultraproduct, it follows that $\G$ has property $PL^p$.
\end{proof}

For the next proposition, the steps of the proof are essentially the same as in the proof of Proposition \ref{coarseemb}. But, in this case, since for a given constant $r$, the trivialization of a fibred coarse embedding is defined on subsets of diameter less than $r$, we need to ''$r$-localize`` our construction of isometric affine actions of the quotient groups $\G/\G_{n_r}$.

\begin{prop}\label{fceplp}
 Let $1\leq p \leq \infty$ and let $\G$ be a finitely generated, residually finite group. If a box space $\Box \G$ of $\G$ admits a fibred coarse embedding into some $L^p$ space, then $\G$ has property $PL^p$.
\end{prop}

\begin{proof}
 The case $p=\infty$ is trivial (see Remark \ref{pinfty}). \\
 Let $1\leq p < \infty$. Let $\{\G_n\}_{n \in \N^*}$ be a nested sequence of finite index normal subgroups of $\G$ with trivial intersection such that the associated box space $\Box \G$ admits a fibred coarse embedding into a $L^p$ space denoted by $B$. \\
 We set $X_n= \G/\G_n$ and $X=\bigsqcup_{n \in \N^*}X_n (=\Box \G)$. Let $r \in \N^*$. By Definition \ref{dffce}, there exist $K_r$ and a trivialization $t_C$ for each $C\subset X\smallsetminus K_r$ of diameter less than $r$ satisfying conditions $i)$ and $ii)$. \smallskip \\
 Now, choose $n_r$ large enough such that $X_{n_r} \subset X \smallsetminus K_r$ and the quotient map $\pi_{n_r}: \G \twoheadrightarrow X_{n_r}$ is $r$-isometric, i.e., for each subset $Y \subset \G$ of diameter less than $r$, $(\pi_{n_r})_{|_{Y}}$ is an isometry onto its image. \\
 For $z \in X_{n_r}$, we denote by $C_z:=\{x \in X_{n_r} \st d_{X_{n_r}}(z,x) < r\}$ the $r$-ball centered in $z$ of $X_{n_r}$ and we set, for $x \in X_{n_r}$, the following vector $c^z_r(x)$ of $B$:
 \begin{center}\begin{equation*}
    c^z_r(x):=\begin{cases} t_{C_z}(z)(s(z))-t_{C_z}(zx)(s(zx))&\text{ if } d_{X_{n_r}}(e,x) < r \text{ (i.e }x \in C_e\text{)}; \\
                                0&\text{ otherwise,}
                            \end{cases}
    \end{equation*}
   \end{center}
 where $e$ is the identity element of $X_{n_r}$.
 Notice that, by Definition \ref{dffce} 3. i) for any $z \in X_{n_r}$ and any $x \in C_e$, $\rho_1(d_{X_{n_r}}(e,x))\leq \|c^z_r(x)\|_B \leq \rho_2(d_{X_{n_r}}(e,x))$. \smallskip \\
 Let us consider the map $\tilde{b}_r : X_{n_r} \rightarrow \bigoplus_{z \in X_{n_r}}B$, defined by, for $x \in X_{n_r}$:
 $$ \tilde{b}_r(x)=\frac{1}{(\# X_{n_r})^{\frac{1}{p}}}\bigoplus_{z \in X_{n_r}}c^z_r(x). $$
 We endow $\bigoplus_{z \in X_{n_r}}B$ with the norm induced by the norm of $\ell_p^n$ i.e. for $\k=\bigoplus_{z \in X_{n_r}} \k_z$, $$ \|\k\|_p=\left(\sum_{z \in X_{n_r}}\|\k_z\|_B^p\right)^{\frac{1}{p}}.$$ \smallskip
 Hence, for $x \notin C_e$, $\tilde{b}_r(x)$ vanishes, and for $x \in C_e$, $\rho_1(d_{X_{n_r}}(e,x))\leq \|\tilde{b}_r(x))\|_p \leq \rho_2(d_{X_{n_r}}(e,x))$. \smallskip \\
 We claim that $\tilde{b}_r(x)$ is a $r$-locally cocycle for a $r$-locally isometric representation $\tilde{\s}_r$ that we define as follows: \\
 For $x \in C_e$ and $z \in X_{n_r}$, let $\rho_{C_zC_{zx}}$ be the linear part of the affine isometry $t_{C_zC_{zx}}:B \rightarrow B$. We define $\tilde{\s}_r(x): \bigoplus_{z \in X_{n_r}}B \rightarrow \bigoplus_{z \in X_{n_r}}B$ by, for $\k=\bigoplus_{z \in X_{n_r}}\k_z$:
 \begin{center}\begin{equation*}
    \tilde{\s}_r(x)(\k):=\begin{cases} \bigoplus_{z \in X_{n_r}}\rho_{C_zC_{zx}}(\k_{zx})&\text{ if } x \in C_e; \\
                                \k&\text{ otherwise.}
                            \end{cases}
    \end{equation*}
   \end{center}
 The map $\tilde{\s}_r$ is indeed a $r$-locally isometric representation: it is clear that $\tilde{\s}_r(x)$ is a isometric isomorphism for all $x \in X_{n_r}$; moreover it follows from Definition \ref{dffce} 3. ii) that $t_{C_zC_{zy}}\circ t_{C_{zy}C_{zyx}}=t_{C_zC_{zyx}}$ for all $x,y \in C_e$ with $d_{X_{n_r}}(e,yx)<r$, and then, $\rho_{C_zC_{zy}}\circ \rho_{C_{zy}C_{zyx}}=\rho_{C_zC_{zyx}}$. Hence, $\tilde{\s}_r(yx)=\tilde{\s}_r(y)\tilde{\s}_r(x)$. \smallskip \\
 Now, we have, for $x,y \in C_e$ with $d_{X_{n_r}}(e,yx)<r$, $ \tilde{\s}_r(y)(\tilde{b}_r(x))+\tilde{b}_r(y)=\tilde{b}_r(yx)$. In fact, by noticing that for an affine isometry $T$ with linear part $\rho$, $\rho(x-y)=Tx-Ty$, we have:
 \begin{center}
  \begin{tabular}{rl}
   $\rho_{C_zC_{zy}}(c_r^{zy}(x)) $&$=\rho_{C_zC_{zy}}\left(t_{C_{zy}}(zy)(s(zy))-t_{C_{zy}}(zyx)(s(zyx))\right),$ \smallskip \\
   &$=t_{C_zC_{zy}}\circ t_{C_{zy}}(zy)(s(zy))-t_{C_zC_{zy}}\circ t_{C_{zy}}(zyx)(s(zyx))+t_{C_z}(z)(s(z))-t_{C_z}(zy)(s(zy)),$ \smallskip\\
   $\rho_{C_zC_{zy}}(c_r^{zy}(x))$&$=t_{C_{z}}(zy)(s(zy))-t_{C_{z}}(zyx)(s(zyx))$ \smallskip\\
  \end{tabular}
 \end{center}
 since $t_{C_zC_{zy}}\circ t_{C_{zy}}(zy)=t_{C_z}(zy)$ (by Definition \ref{dffce} 3. ii)) . \smallskip \\
 Thus, $$\rho_{C_zC_{zy}}(c_r^{zy}(x))+c_r^{z}(y)=t_{C_{z}}(z)(s(z))-t_{C_{z}}(zyx)(s(zyx))=c_r^{z}(yx).$$ 
 It follow that: 
 $$\tilde{\s}_r(y)(\tilde{b}_r(x))+\tilde{b}_r(y)=\frac{1}{(\# X_{n_r})^{\frac{1}{p}}}\bigoplus_{z \in X_{n_r}}\left(\rho_{C_zC_{zy}}(c_r^{zy}(x))+c_r^z(y)\right)=\tilde{b}_r(yx)$$
 which proves our claim.
 \medskip  \\
 Now, let $\s_r:=\tilde{\s}_r \circ \pi_{n_r}$ and $b_r=\tilde{b}_r \circ \pi_{n_r}$ be the lifts of $\tilde{\s}_r$ and $\tilde{b}_r$ to the $r$-ball $\{g \in \G \st d_{\G}(e_{\G},g)<r\}$ of $\G$ and define $\s_r =Id$, $b_r=0$ outside the $r$-ball of $\G$. Then $\s_r$ is a $r$-locally isometric representation action of $\G$ on $\bigoplus_{X_{n_r}}B$, $b_r$ is a $r$-locally cocycle with respect to $\s_r$. Then the map $\a_r$ such that $\a_r(g)\LargerCdot := \s_r(g)\LargerCdot +b_r(g)$ is a $r$-locally isometric affine action of $\G$ on $\bigoplus_{X_{n_r}}B$ and we have, for $g \in \G$ with $d_{\G}(e_{\G},g)<r$:
 $$ \rho_1(d_{\G}(e_{\G},g))\leq \|b_r(g)\|_p \leq \rho_2(d_{\G}(e_{\G},g)). $$ \medskip
 
 From these local isometric affine actions, we build a global isometric affine action of $\G$ thanks to Lemma \ref{rlocultraprod}. \\ 
 Let $\UU$ be a non principal ultrafilter on $\N^*$, and let $B_{\UU}$ be the ultraproduct of the family $\left(\bigoplus_{X_{n_r}}B\right)_{r \in \N^*}$ with respect to $\UU$. For each $r \in \N^*$, $\a_r$ is a $r$-locally isometric affine action of $\G$ on $\bigoplus_{X_{n_r}}B$ and since, for any $g \in \G$, $\|b_r(g)\|_p \leq \rho_2(d_{\G}(e_{\G},g))$ for all $r \in \N^*$, $(b_r(g))_{r \in \N^*}$ belongs to $B_{\UU}$. Hence, by Lemma \ref{rlocultraprod}, there exists an isometric affine action $\a$ of $\G$ on $B_{\UU}$ such that $b: g \mapsto (b_r(g))$ is a cocycle with respect to the linear part of this action. Moreover, since for any $g \in \G$, $ \rho_1(d_{\G}(e_{\G},g))\leq \|b_r(g)\|_p$ for all $r$ large enough, we have, for all $g \in \G$:
 $\rho_1(d_{\G}(e_{\G},g))\leq\|b(g)\|_{B_{\UU}},$
 and thus, $\a$ is proper. \smallskip \\
 As the class of $L^p$ spaces is closed under $p$-normed powers and ultraproduct, it follows that $\G$ has property $PL^p$.
 \end{proof}

 \begin{proof}[Proof of Theorem \ref{maintheo}]
  It follows from Corollary \ref{cwwcor} and Proposition \ref{fceplp}.
 \end{proof}

 \bibliographystyle{alpha} 
 \bibliography{biblio/biblio}

\end{document}